\documentclass[11pt]{article}
\usepackage[utf8]{inputenc}

\usepackage[utf8]{inputenc}
\usepackage{ mathrsfs }
\usepackage{ dsfont }
\usepackage{amsmath}
\usepackage{indentfirst}
\usepackage{mathtools}
\usepackage{amsfonts}
\usepackage{amsthm}
\usepackage{graphicx}
\usepackage[english]{babel}
\usepackage{fullpage}
\usepackage{enumerate}
\usepackage{tasks}
\usepackage{color}
\usepackage{theoremref}
\usepackage{relsize}
\usepackage{comment}
\usepackage{mathrsfs}
\usepackage{ amssymb }
\usepackage{ textcomp }
\usepackage{pgfplots}
\usepackage{tikz}
\usepackage{ifthen}             % For tikz-hypergraph
\usetikzlibrary{calc}           % For tikz-hypergraph
\usetikzlibrary{patterns}
\usepackage{pgf}
\usepackage{tikz}
\usetikzlibrary{graphs}
\usetikzlibrary{positioning}
\pgfdeclarelayer{bg}    % declare background layer
\pgfdeclarelayer{fg}    % declare foreground layer
\pgfsetlayers{bg,main,fg}  % set the order of the layers (main is the standard layer)

\def\rotateclockwise#1{
	\newdimen\xrw
	\pgfextractx{\xrw}{#1}
	\newdimen\yrw
	\pgfextracty{\yrw}{#1}
	\pgfpoint{\yrw}{-\xrw}
}

\def\rotatecounterclockwise#1{
	\newdimen\xrcw
	\pgfextractx{\xrcw}{#1}
	\newdimen\yrcw
	\pgfextracty{\yrcw}{#1}
	\pgfpoint{-\yrcw}{\xrcw}
}

\def\outsidespacerpgfclockwise#1#2#3{
	\pgfpointscale{#3}{
		\rotateclockwise{
			\pgfpointnormalised{
				\pgfpointdiff{#1}{#2}}}}
}

\def\outsidespacerpgfcounterclockwise#1#2#3{
	\pgfpointscale{#3}{
		\rotatecounterclockwise{
			\pgfpointnormalised{
				\pgfpointdiff{#1}{#2}}}}
}

\def\outsidepgfclockwise#1#2#3{
	\pgfpointadd{#2}{\outsidespacerpgfclockwise{#1}{#2}{#3}}
}

\def\outsidepgfcounterclockwise#1#2#3{
	\pgfpointadd{#2}{\outsidespacerpgfcounterclockwise{#1}{#2}{#3}}
}

\def\outside#1#2#3{
	($ (#2) ! #3 ! -90 : (#1) $)
}

\def\cornerpgf#1#2#3#4{
	\pgfextra{
		\pgfmathanglebetweenpoints{#2}{\outsidepgfcounterclockwise{#1}{#2}{#4}}
		\let\anglea\pgfmathresult
		\let\startangle\pgfmathresult
		
		\pgfmathanglebetweenpoints{#2}{\outsidepgfclockwise{#3}{#2}{#4}}
		\pgfmathparse{\pgfmathresult - \anglea}
		\pgfmathroundto{\pgfmathresult}
		\let\arcangle\pgfmathresult
		\ifthenelse{180=\arcangle \or 180<\arcangle}{
			\pgfmathparse{-360 + \arcangle}}{
			\pgfmathparse{\arcangle}}
		\let\deltaangle\pgfmathresult
		
		\newdimen\x
		\pgfextractx{\x}{\outsidepgfcounterclockwise{#1}{#2}{#4}}
		\newdimen\y
		\pgfextracty{\y}{\outsidepgfcounterclockwise{#1}{#2}{#4}}
	}
	-- (\x,\y) arc [start angle=\startangle, delta angle=\deltaangle, radius=#4]
}

\def\corner#1#2#3#4{
	\cornerpgf{\pgfpointanchor{#1}{center}}{\pgfpointanchor{#2}{center}}{\pgfpointanchor{#3}{center}}{#4}
}

\def\hedgeiii#1#2#3#4{
	\outside{#1}{#2}{#4} \corner{#1}{#2}{#3}{#4} \corner{#2}{#3}{#1}{#4} \corner{#3}{#1}{#2}{#4} -- cycle
}

\def\hedgem#1#2#3#4{
	\outside{#1}{#2}{#4}
	\pgfextra{
		\def\hgnodea{#1}
		\def\hgnodeb{#2}
	}
	foreach \c in {#3} {
		\corner{\hgnodea}{\hgnodeb}{\c}{#4}
		\pgfextra{
			\global\let\hgnodea\hgnodeb
			\global\let\hgnodeb\c
		}
	}
	\corner{\hgnodea}{\hgnodeb}{#1}{#4}
	\corner{\hgnodeb}{#1}{#2}{#4}
	-- cycle
}

\def\hgrotate#1{
	\newdimen\x
	\pgfextractx{\x}{#1}
	\newdimen\y
	\pgfextracty{\y}{#1}
	\pgfpoint{-\y}{\x}
}

\def\hgperpr#1#2#3{
	\pgfpointscale{#3}{
		\hgrotate{
			\pgfpointnormalised{
				\pgfpointdiff{#1}{#2}}}}
}

\def\hgaddeperpr#1#2#3{
	\pgfpointadd{#2}{\hgperpr{#1}{#2}{#3}}
}

\def\hgaddsperpr#1#2#3{
	\pgfpointadd{\hgperpr{#1}{#2}{#3}}{#1}
}

\def\hgcorner#1#2#3#4{
	\pgflineto{\hgaddeperpr{#1}{#2}{#4}}
	\pgfmathanglebetweenpoints{#1}{#2}\let\anga\pgfmathresult
	\pgfpatharcto{#4}{#4}{90 + \anga}{0}{0}{\hgaddsperpr{#2}{#3}{#4}}
}

\newtheorem{theorem}{Theorem}[section]
\newtheorem{lemma}[theorem]{Lemma}

\newtheorem{corollary}[theorem]{Corollary}

\newtheorem{claim}[theorem]{Claim}

\newtheorem{conjecture}[theorem]{Conjecture}

\theoremstyle{remark}

\def\C{\mathcal{C}}

\def\F{\mathcal{F}}

\def\P{\mathcal{P}}

\def\S{\mathcal{S}}

\def\se{\subseteq}
\def\ex{\textup{ex}}

\newcommand{\eps}{\varepsilon}

\definecolor{lblue}{rgb}{0.5,0.5,1}

\newcommand{\eq}[1]{\begin{equation}\label{eq:#1}}
	\newcommand{\eqe}{\end{equation}}

\newcommand\restr[2]{{
		\left.\kern-\nulldelimiterspace % automatically resize the bar with \right
		#1 % the function
		\vphantom{\big|} % pretend it's a little taller at normal size
		\right|_{#2} % this is the delimiter
}}

\author{%
	J\'ozsef Balogh \footnote{Department of Mathematics, University of Illinois at Urbana-Champaign, Urbana, Illinois 61801, USA, and Moscow Institute of Physics and Technology, Russian Federation. E-mail: \texttt{jobal@illinois.edu}. Research is partially supported by NSF grants DMS-1764123 and RTG DMS-1937241, the Arnold O. Beckman Research
		Award (UIUC Campus Research Board RB 18132), the Langan Scholar Fund (UIUC), and the Simons Fellowship.}
	\and Felix Christian Clemen \footnote {Department of Mathematics, University of Illinois at Urbana-Champaign, Urbana, Illinois 61801, USA, E-mail: \texttt{fclemen2@illinois.edu}. Research is partially supported by the Arnold O. Beckman Research
		Award (UIUC Campus Research Board RB 18132).}
	\and  Let\'icia Mattos \footnote {Freie Universität Berlin and Berlin Mathematical School (BMS/MATH+), E-mail: \texttt{leticiadmat@gmail.com}.
		Research is funded by the Deutsche Forschungsgemeinschaft (DFG, German Research Foundation) under Germany's Excellence Strategy -- The Berlin Mathematics Research Center MATH+ (EXC-2046/1, project ID: 390685689).}
}
\title{Counting $r$-graphs without forbidden configurations}
\date{}

\begin{document}
	\maketitle
	\begin{abstract}
		One of the major problems in combinatorics is to determine the number of $r$-uniform hypergraphs ($r$-graphs) on $n$ vertices which are free of certain forbidden structures.
		This problem dates back to the work of Erd\H{o}s, Kleitman and Rothschild, who showed that the number of $K_r$-free graphs on $n$ vertices is $2^{\ex(n,K_r)+o(n^2)}$.
		Their work was later extended to forbidding graphs as induced subgraphs by Prömel and Steger.
%		The results were stated in terms of a different notion of extremal number, which is not very well-understood even for simple families of $3$-graphs.
		
		Here, we consider one of the most basic counting problems for $3$-graphs. Let $E_1$ be the $3$-graph with $4$ vertices and $1$ edge. What is the number of induced $\{K_4^3,E_1\}$-free $3$-graphs on $n$ vertices? We show that the number of such $3$-graphs is of order $n^{\Theta(n^2)}$. More generally, we determine asymptotically the number of induced $\F$-free $3$-graphs on $n$ vertices for all families $\F$ of $3$-graphs on $4$ vertices. 
		We also provide upper bounds on the number of $r$-graphs on $n$ vertices which do not induce $i \in L$ edges on any set of $k$ vertices, where $L \se \big \{0,1,\ldots,\binom{k}{r} \big\}$ is a list which does not contain $3$ consecutive integers in its complement.
		Our bounds are best possible up to a constant multiplicative factor in the exponent when $k = r+1$.
		The main tool behind our proof is counting the solutions of a constraint satisfaction problem.
	\end{abstract}
	\section{Introduction}
	\subsection{History}
	For an $r$-uniform hypergraph ($r$-graph) $F$, let $\textup{ex}(n,F)$ denote the maximum number of edges in an $F$-free $r$-graph on $n$ vertices.
	One of the central questions in extremal combinatorics is to determine the extremal number $\textup{ex}(n,F)$.
	For $r = 2$, the extremal number is well-understood for all non-bipartite graphs, see~\cite{ErdosStone} and~\cite{Turan}. However, 
	determining the extremal number for general $r$-graphs is a well-known and hard problem. 
	The simplest and still not answered question posed by Tur\'an asks to determine the extremal number of $K_4^3$, the complete $3$-graph on $4$ vertices. 
	It is widely believed that 
	\begin{align*}
		\textup{ex}(n,K_4^3)=\left(\dfrac{5}{9}+o(1)\right)\binom{n}{3}.
	\end{align*}
	In a series of papers, different $K_4^3$-free $3$-graphs on $n$ vertices and $\frac{5}{9}\binom{n}{3}+o(n^3)$ edges were constructed by Brown~\cite{K43brown}, Kostochka~\cite{K43Kostochka} and Fon-der-Flaass~\cite{K43Fonderflaass} and Razborov~\cite{RazbarovK43}.
	In 2008, Frohmader~\cite{MR2465761} showed that there are $\Omega(6^{n/3})$
	non-isomorphic $r$-graphs which are conjectured to be extremal.
	This is believed to be one of the reasons of the difficulty of this problem. 
	For other related papers, see~\cite{BalCleLidmain,PikhurkoK43,RazbarovK43}. 
	
	The problem of determining the extremal number can also be extended to families of induced $r$-graphs.
	For a family of $r$-graphs $\F$, let $\ex_I(n,\F)$ denote the maximum number of edges in an induced $\F$-free $r$-graph on $n$ vertices.
	In 2010, Razborov~\cite{RazbarovK43} used the method of flag algebras to determine $\ex_I(n,\{K_4^3,E_1\})$, where $E_1$ denotes the $3$-graph with $4$ vertices and $1$ edge. 
	In his paper, he showed that
	\begin{align*}
		\ex_I(n,\{K_4^3,E_1\}) = \left(\dfrac{5}{9}+o(1)\right)\binom{n}{3}.
	\end{align*}
	Later, this result was extended by Pikhurko~\cite{PikhurkoK43}, who obtained the corresponding stability result and proved that there is only 
	one extremal induced $\{K_4^3,E_1\}$-free $3$-graph on $n$ vertices, up to isomorphism.
	Sometimes referred to as Tur\'an's construction and here denoted by $C_n$, the extremal induced $\{K_4^3,E_1\}$-free $3$-graph on $[n]$ is obtained as follows.
	Let $V_1\cup V_2\cup V_3$ be a partition of $[n]$ with $\big | |V_i|-|V_j| \big |\le 1$ for all $i,j \in [3]$. 
	An edge is placed in $C_n$ if it intersects each of the classes $V_1$, $V_2$ and $V_3$, or if for some $i \in [3]$ it contains two elements of $V_i$ and one of $V_{i+1}$, where the indices are understood modulo $3$. See Figure~\ref{fig:Cn} for an illustration of $C_n$. 		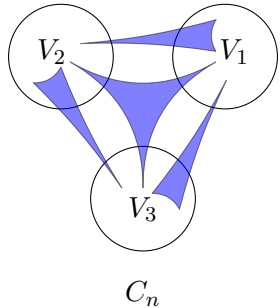
\begin{figure}[ht]
		\begin{center}	
\tikzset{
vtx/.style={inner sep=1.1pt, outer sep=0pt, circle, fill,draw}, 
hyperedge/.style={fill=blue,opacity=0.5,draw=black}, 
vtxBig/.style={inner sep=14pt, outer sep=0pt, circle, fill=white,draw}, 
hyperedge/.style={fill=blue,opacity=0.5,draw=black}, 
}
\begin{tikzpicture}[scale=1.4]
\draw (30:0.9) coordinate(x1) node[vtxBig]{};
\draw (150:0.9) coordinate(x2) node[vtxBig]{};
\draw (270:0.9) coordinate(x3) node[vtxBig]{};
\draw
(30:0.8) coordinate(1) %node[vtx](a){}
(150:0.8) coordinate(2) %node[vtx](b){}
(270:0.8) coordinate(3) %node[vtx](c){}
;
\draw[hyperedge] (1) to[out=210,in=330] (2) to[out=330,in=90] (3) to[out=90,in=210] (1);
\foreach \ashift in {30,150,270}{
\draw
(\ashift:0.8)++(\ashift+60:0.1) coordinate(1) %node[vtx](a){}
++(\ashift+60:0.3) coordinate(2) %node[vtx](a){}
(\ashift+120:0.8)++(\ashift+120+90:-0.2) coordinate(3)% node[vtx](a){}
;
\draw[hyperedge] (1) to[bend left] (2) to[bend left=5] (3) to[bend left=5] (1);
}
\draw (30:1)  node{$V_1$};
\draw (150:1) node{$V_2$};
\draw (270:1) node{$V_3$};
\draw (0,-1.8) node{$C_n$};
\end{tikzpicture}
\caption{Illustration of $C_n$.}\label{fig:Cn}
\end{center}
\end{figure}
	
	In this paper, 
	we first consider the problem of counting $\{K_4^3,E_1\}$-free $3$-graphs on $n$ vertices, which is the counting problem related to the results of Razborov~\cite{RazbarovK43} and Pikhurko~\cite{PikhurkoK43}.
	Recently, Balogh and Mubayi~\cite{personalcomBM} observed that a standard application of the hypergraph container method~\cite{MR3327533, MR3385638} shows that the number of induced $\{K_4^{3}, E_1\}$-free graphs on $n$ vertices is $2^{O(n^{8/3})}$.
	From the other side, we can construct a family $Q(n)$ with $2^{\Omega(n^2\log n)}$
	subgraphs of $C_n$ which are $\{K_4^3,E_1\}$-free. 
	A $3$-graph is in $Q(n)$ if it is obtained from $C_n$ by removing a linear\footnote{A $3$-graph $H$ is linear if every pair of distinct edges $e_1, e_2 \in E(H)$ satisfies $|e_1 \cap e_2| \le 1$.} $3$-graph with the additional property that every edge contains one element from each of the classes $V_1$, $V_2$ and $V_3$.
	It is not hard to show that every $3$-graph in $Q(n)$ is in fact $\{K_4^3,E_1\}$-free and that $|Q(n)| = 2^{\Omega(n^2\log n)}$.
	Balogh and Mubayi~\cite{personalcomBM} conjectured that almost all $\{K_4^3,E_1\}$-free 3-graphs are in this family, up to isomorphism. 
	\begin{conjecture}[Balogh and Mubayi~\cite{personalcomBM}]
		\label{conj:14-freecounting}
		Almost all $\{K_4^3,E_1\}$-free $3$-graphs on $[n]$ are in $Q(n)$, up to isomorphism. 
	\end{conjecture}
	The motivation behind this conjecture comes from similar results. In particular, Person and Schacht~\cite{MR2809321} proved that almost all Fano-plane free $3$-graphs are bipartite, and Balogh and Mubayi~\cite{MR2927636} proved that almost all $F_5$-free triple systems are tripartite, where $F_5$ is the $5$-vertex $3$-graph with edge set $\{123,124,345\}$. See also \cite{MR2771596} for results along the same line.

	%The so-called \emph{degenerate} case, when a bipartite graph is forbidden, seems to be harder, and there are many milestone results, such as \cite{MR2776642,MR4073190,MR676877,MR3505750}. Similar to the bounds we will be presenting, these bounds are best possible up to multiplicative constant factors in the exponent. The degenerate cases for hypergraphs also got some attention, see \cite{MR3906641,Spiro}.

	The problem of counting $r$-graphs which are free of forbidden structures dates back to the work of Erd\H{o}s, Kleitman and Rothschild~\cite{EKR} in the context of graphs.
	They showed that the number of $K_s$-free graphs on $n$ vertices is $2^{(1+o(1))\ex(n,K_s)}$.
	Their work was later extended to all non-bipartite graphs by Erd\H{o}s, Frankl and Rödl~\cite{EFR} using the Szemerédi regularity lemma. 
	For other related results, see~\cite{MR3906641,MR2776642,EFR,MR4073190,MR676877,MR3505750,Spiro}.
	In a sequence of papers~\cite{PromelSteger1, MR1139486, PromelSteger2}, Prömel and Steger considered the corresponding problem for induced graphs.
	Their results were stated in terms of a different notion of extremal number, which was latter generalised by Dotson and Nagle~\cite{MR2482964} as follows.
	Given a family of $r$-graphs $\mathcal{F}$, let $M$ and $N$ be sets in $[n]:= \{1,\ldots,n\}$ with the following properties: (i) $M \cap N = \emptyset$; and (ii) for $G \se \binom{[n]}{r}\setminus (M \cup N)$, the $r$-graph $G \cup M$ is induced $\F$-free.
	The notation $\binom{[n]}{r}$ stands for $\{S \se [n]: |S| = r\}$. 
	The $*$-extremal number $\ex^{*}(n,\F)$ is defined as
	\begin{align*}
		\ex^{*}(n,\F) := \binom{n}{3}-\min \limits_{M,N} \Big (|M|+|N| \Big),
	\end{align*}
	where the minimum is over all sets $M, \, N \se [n]$ satisfying (i) and (ii).
	%It is not hard to see that the number of induced $\F$-free graphs on $n$ vertices is at least $2^{\ex^{*}(n,\F)}$.
	In 1992, Prömel and Steger~\cite{MR1139486} showed that the number of induced $F$-free graphs on $n$ vertices is $2^{\ex^{*}(n,\F)+o(n^2)}$. This result was later extended by  Alekseev~\cite{MR1181539} and Bollobás and Thomason~\cite{MR1425205} for families of graphs, and by Kohayakawa, Nagle and Rödl~\cite{MR1967402} for $3$-graphs.
	In 2009, Dotson and Nagle~\cite{MR2482964} 
	generalized these results, showing that for all families of $r$-graphs $\F$ the number of induced $\F$-free $r$-graphs is $2^{\ex^{*}(n,\F)+o(n^r)}$.
	
	For a family $\F$ of $r$-graphs such that $\ex^{*}(n,\F)=o(n^{r})$, the counting results mentioned above are not precise.
	In the case of graphs, Alon, Balogh, Bollob\'as and Morris~\cite{MR2763071}
	obtained a more refined result. They showed that the number of induced $\F$-free graphs on $n$ vertices is $2^{\ex^{*}(n,\F)+O(n^{2-\eps})}$, where $\varepsilon>0$ depends only on the family $\mathcal{F}$. Terry~\cite{MR3787371} generalized this result to finite relational languages which in particular covers $r$-graphs. For a family of $r$-graphs $\mathcal{F}$, her result says that the number of induced $\mathcal{F}$-free 3-graphs is either either $2^{\Theta(n^r)}$ or there exists $\varepsilon > 0$ such that for all large enough $n$, the number of induced $\mathcal{F}$-free 3-graphs is at most $2^{n^{r-\varepsilon}}$. 
%	Our next theorem improves the results in~\cite{MR1181539, MR2763071, MR1425205, MR2482964, MR1967402} for some families of $r$-graphs.
%	Our bounds are best possible in the exponent when the family of $r$-graphs considered has $r+1$ vertices,
%	see Section~\ref{counting:sharpness} for more details.

	 %It follows from Theorem~\ref{maincounting} that there are surprisingly few such induced $\{E_1,K_4^3\}$-free $3$-graphs. This indicates that we are still far away from solving Tur\'an's tetrahedron problem. 
\subsection{Our results}
Our first theorem determines the number of $\{K_4^{3}, E_1\}$-free graphs up to a constant factor on the exponent, making progress towards Conjecture~\ref{conj:14-freecounting}.
	 \begin{theorem}
	 	\label{maincountingcorl}
	 	The number of $\{K_4^{3}, E_1\}$-free $3$-graphs on $n$ vertices is
	 	$2^{\Theta(n^2 \log n)}$.
	 \end{theorem}

	 More generally, we also determine the number of induced $\F$-free $3$-graphs on $n$ vertices for all families $\F$ of $3$-graphs on $4$ vertices. 
	 Since any $3$-graph on $4$ vertices is determined by its number of edges,
	 our result is stated in terms of forbidden number of edges.
	 For a set $L \se \{0,1,2,3,4\}$, let $f(n,3,4,L)$ be the number of $3$-graphs on $n$ vertices which do not induce $i \in L$ edges on any set of $4$ vertices.
	 Our result can be stated as follows, where we do not attempt to optimize the constants in the exponent.
	 
	\begin{theorem}\label{theo:tablevalues}
		Let $L \se \{0,1,2,3,4\}$ be a set. Then, the following holds for $n \ge 13$.
		\begin{enumerate}
			\item [$(a)$] If $\{0,4\} \se L$ or $\{1,2,3\} \se L$, then $f(n,3,4,L) \in \{0,1,2\};$

			\item [$(b)$] If $L = \{0,2,3\}$ or $L = \{1,2,4\}$, then $f(n,3,4,L) = n+1;$

			\item [$(c)$] If $L = \{0,1,3\}$ or $L = \{1,3,4\}$, then $f(n,3,4,L) = 2^{\Theta(n\log n)};$
			
			\item [$(d)$] If $L = \{1,3\}$, then $f(n,3,4,L) = 2^{\binom{n-1}{2}};$

			\item [$(e)$] If $L \in \Big\{\emptyset, \{0\}, \{1\}, \{3\}, \{4\}, \{0,1\}, \{3,4\} \Big\}$, then $f(n,3,4,L) = 2^{\Theta(n^3)};$

			\item [$(f)$] For all the remaining cases, we have $f(n,3,4,L) = 2^{\Theta(n^2 \log n)}.$
		\end{enumerate}
	\end{theorem}
Note that some of the statements in Theorem~\ref{theo:tablevalues} are trivial and others are known. We included those for the sake of completeness.

	Before we state our next theorem, we need some notation.
	Let $k, r \in \mathbb{N}$ and $L\subseteq \big\{0,1,\ldots,\binom{k}{r}\big\}$ be a set, which we refer as a \emph{list}.
	We say that an $r$-graph $G$ is \emph{$(L,k)$-free} if for all $i \in L$ there is no set of $k$ vertices in $G$ inducing $i$ edges.
	By generalising our previous notation, we denote by $f(n,r,k,L)$ the number of $(L,k)$-free $r$-graphs on $n$ vertices.
	Our next theorem extends Theorem~\ref{maincountingcorl} to $r$-graphs and \emph{$3$-good lists}.
	We say that a list $L$ is \emph{$3$-good} if $\{i,i+1,i+2\} \cap L \neq \emptyset$ for all $i \in  \big\{0,1,\ldots,\binom{k}{r}-2\big\}$. That is, the complement of $L$ does not contain $3$ consecutive integers. Throughout this paper, all logarithms are in base 2.
		\begin{theorem}
		\label{maincounting}
		Let $n \ge k>r\ge 2$ be integers and $L\subseteq \big\{0,1,\ldots, \binom{k}{r}\big\}$ be a list. If $L$ is $3$-good, then \begin{align*}
			f(n,r,k,L)\leq 2^{2kn^{r-1}+n^{r-1}\log n}. 
		\end{align*}
	\end{theorem}

	%Note that $f(n,r,k,L)=f(n,r,k,L^c)$.
	
	The main tool behind the proof of Theorem~\ref{maincounting} is a lemma which counts the number of solutions of a certain constraint satisfaction problem, see Lemma~\ref{CSPcount}. 
	For $L=\{2,3,\ldots,r+1\}$, we observe that $f(n,r,r+1,L)$ is equal to the number of $r$-graphs such that, for every pair of edges, the size of their intersection is not $r-1$.
	This is related to the problem of counting designs, a heavily studied object in combinatorics, see~\cite{MR3779688,MR2961567,MR3124689}.

	The rest of this paper is organized as follows. In Section~\ref{counting:sharpness} we discuss the sharpness of Theorem~\ref{maincounting}; in Section~\ref{sec:maincount} we present the proof of Theorem~\ref{maincounting}; in Section~\ref{sec:tableexplanation} we prove Theorems~\ref{maincountingcorl} and~\ref{theo:tablevalues}.
	
	\section{Sharpness discussion of Theorem~\ref{maincounting}}
	\label{counting:sharpness}
	In this section, we provide three examples which show that Theorem~\ref{maincounting} is sharp for $k=r+1$.
	Our first lemma shows that there is a $3$-good list that achieves the upper bound given by Theorem~\ref{maincounting}.

	\begin{lemma}
		\label{listsharp1}
		For $r\geq 2$ we have
		\begin{align*}
			f(n,r,r+1,\{2,3,\ldots,r+1\})=2^{\Theta(n^{r-1}\log n)}.
		\end{align*}
	\end{lemma}
	
	\begin{proof}
		The list $L := \{2,3,\ldots,r+1\}$ is $3$-good and therefore Theorem~\ref{maincounting} can be applied, which gives the upper bound.
		Now, let $M(n,r)$ be the set of $r$-graphs on $[n]$ such that every $(r-1)$-subset\footnote{A $t$-set or $t$-subset is a set with exactly $t$ elements.} of $[n]$ is contained in at most one edge. 
		Note that the number of $(L,r+1)$-free $r$-graphs on $[n]$ is equal to $|M(n,r)|$. Thus, to lower bound $f(n,r,r+1,L)$ it suffices to show that $|M(n,r)|=2^{\Omega(n^{r-1}\log n)}$.

		One way to lower bound $M(n,r)$ is to use a result of Keevash~(Theorem 6.1 in \cite{MR3779688}) on \emph{designs}.
		An $r$-graph $G$ on $[n]$ is an $(n, r, t, \lambda)$-\emph{design} if every $t$-subset of $[n]$ is contained in exactly $\lambda$ elements of $E(G)$. 
		If certain divisibility conditions involving $n$ and $r$ are satisfied, then Keevash's result implies that the number of $(n,r,r-1,1)$-designs is $2^{\Theta(n^{r-1}\log n)}$. In particular, we have $|M(n,r)|\geq 2^{\Omega(n^{r-1}\log n)}$.
		For the sake of self-completeness, we present here another simple way to derive this inequality. To do so, we build a subfamily of $r$-graphs in $M(n,r)$ via the following greedy procedure. 
		Let $e_1$ be an $r$-subset in $[n]$.
		For $i>1$, let $e_i$ be an $r$-subset in $[n]$ such that $|e_i \cap e_j| \neq r-1$ for all $j \in [i-1]$.
		The procedure stops when an edge $e_i$ with this property cannot be found.
		As there are at most $rni$ sets of size $r$ which intersect some element of $\{e_1,\ldots, e_{i-1}\}$ in exactly $r-1$ vertices, we have at least $\binom{n}{r}-nri$ choices for $e_i$.
		It follows that the procedure lasts for at least $\frac{n^{r-1}}{2r^{r+1}}$ steps.
		As we have $\binom{n}{r}-nri \ge \frac{1}{2}\binom{n}{r}$ for all $i \le \frac{n^{r-1}}{2r^{r+1}}$, it follows that the number of $r$-graphs in $M(n,r)$ is at least
		\begin{align*}
		|M(n,r)|\geq 	\frac{\left(\frac{1}{2}\binom{n}{r}\right)^{\frac{n^{r-1}}{2r^{r+1}}}}{\left(\frac{n^{r-1}}{2r^{r+1}}\right)!} \geq \left(\frac{r^{r+1}\binom{n}{r}}{n^{r-1}}  \right)^{\frac{n^{r-1}}{2r^{r+1}}} \geq n^{\frac{n^{r-1}}{2r^{r+1}}}=2^{\Theta(n^{r-1}
				\log n)}.
		\end{align*}
		 The factorial term above takes the double counting into consideration.
		 Combining this bound with the upper bound from Theorem~\ref{maincounting}, we obtain $f(n,r,r+1,L)=2^{\Theta(n^{r-1}\log n)}$.
	\end{proof}

Our next example shows that there is a list $L$ which is not $3$-good such that the bound presented in Theorem~\ref{maincounting} does not hold.
	\begin{lemma}
		\label{listsharp2}
		For $r\geq 2$ we have
		\begin{align*}
			f(n,r,r+1,\{3,4,\ldots,r+1\})=2^{\Theta(n^{r})}.
		\end{align*}
	\end{lemma}
	\begin{proof}
		Consider an $r$-partition $(V_i)_{i\in[r]}$ of the vertex set $[n]$
		such that $\big||V_i|-|V_j|\big|\le 1$ for all $i,j \in [r]$.
		Let $G$ be the $r$-graph whose set of edges is given by the $r$-sets with one vertex in each class $V_i$.
		This $r$-graph has $\Theta(n^r)$ edges and every subgraph of it is $(L,r+1)$-free, where $L:=\{3,4,\ldots,r+1\}$. We conclude that
		$f(n,r,r+1,L)=2^{\Theta(n^{r})}$.
	\end{proof}

	When $k = n$, there is a $3$-good list $L\se \left \{0,1,2,\ldots,\binom{n}{r}\right\}$  such that the bound on $f(n,r,n,L)$ given by Theorem~\ref{maincounting} is sharp.
	\begin{claim}
		\label{listsharp3}
		Let $r\geq 2$ and $L$ be the set of all odd integers in $ \left \{0,1,2,\ldots,\binom{n}{r}\right\}$. Then, we have
		\begin{align*}
			f(n,r,n,L)=2^{\Theta(n^{r})}.
		\end{align*}
	\end{claim}		
	
	\begin{proof}
		The number of $(L,n)$-free $r$-graphs on $[n]$ is equal to the number of $r$-graphs on $[n]$ with an even number of edges. Clearly, there are $2^{\Theta(n^r)}$ of those.
	\end{proof}
	\section{Proof of Theorem~\ref{maincounting}}
	\label{sec:maincount}
	
	We will start by proving a combinatorial lemma. To state it we use the language of \emph{constraint satisfaction problems} (CSP). 
	Let $\P$ be the family of all subsets of $\{0,1\} \times \{0,1\}$. We refer to the elements of $\P$ as \emph{constraints}.
	A CSP on $[m]$ is a pair $([m],f)$, where $f: \binom{[m]}{2} \rightarrow \P$ is a function assigning a constraint for each pair of vertices. 
	An \emph{assignment} on $[m]$ is a function $g: [m] \rightarrow \{0,1\}$ which assigns for every vertex $v\in [m]$ an integer (or \emph{color}) from $\{0,1\}$. We say that an assignment $g: [m] \rightarrow \{0,1\}$ is \emph{satisfying} for $([m],f)$ if we have
	$(g(a),g(b)) \notin f(\{a,b\})$ for any pair $a,b\in[m]$ such that $a<b$.

	For a CSP $G = ([m],f)$, let $A(G)$ be the set of satisfying assignments for $G$. 
	Now, define 
	\begin{align*}
		\mathcal{C}:=\bigg\{\Big\{(1,0),(0,1)\Big\},\Big\{(0,0)\Big\},\Big\{(1,1)\Big\}\bigg\}.
	\end{align*}
	Observe that $\C$ is a subset of $\P$. 
	Our next lemma shows that for $f: \binom{[m]}{2} \rightarrow \mathcal{C}$ the CSP $G=([m],f)$ satisfies $|A(G)| \le m+1$.
	This bound is best possible, as we can see by the following example. 
	Let $f: \binom{[m]}{2} \rightarrow \mathcal{C}$ be the function given by $f(\{i,j\}) = \{(1,1)\}$ for all $i < j$ and let $G = ([m],f)$ be the corresponding CSP. 
	A function $g: [m] \to \{0,1\}$ is a satisfying assignment for $G$ if and only if $|g^{-1}(1)|\le 1$. 
	As there are exactly $m+1$ choices of $g^{-1}(1)$ for which $|g^{-1}(1)|\le 1$, we have $|A(G)|=m+1$.

	\begin{lemma}
		\label{CSPcount}
		Let $m \in \mathbb{N}$, $f: \binom{[m]}{2} \rightarrow \mathcal{C}$ and $G = ([m],f)$ be a CSP.
		The size of the set $A(G)$ of satisfying assignments for $G$ is bounded by 
		\begin{align*}
			|A(G)|\leq m+1.
		\end{align*}
	\end{lemma}
	\begin{proof}
		We proceed by induction on $m$. 
		The base case is $m=1$. For a CSP $G$ on one vertex we have 
		$|A(G)|\leq 2$, as there is only one vertex to be assigned with a color from $\{0,1\}$. 
		Now, let $m\geq 2$ and assume that for every $i<m$ and every CSP $G$ on $i$ vertices we have $|A(G)|\le i+1$.

		Fix a function $f: \binom{[m]}{2} \rightarrow \mathcal{C}$ and 
		the CSP given by $G=([m],f)$.
		We partition the set of satisfying assignments into $A(G)=A_0 \cup A_1$, where 
		\begin{align*}
			A_0 := \{g \in A(G):g(m) = 0\}\qquad \text{and} \qquad A_1 := \{g \in A(G):g(m) = 1\}.
		\end{align*}
		We gain extra information about the satisfying assignments in each of these sets based on the constraints given by $f$.
		Note that if $f(\{j,m\})=\{(1,0),(0,1)\}$, then we automatically have $g(j)=i$ for all $g \in A_i$ and $i \in \{0,1\}$.
		If $f(\{j,m\})=\{(i,i)\}$, then we must have $g(j)=1-i$ for all $g \in A_i$ and $i \in \{0,1\}$.
		This motivates us to consider the following subsets of $[m]$:
		\begin{align*}
			F_0 := \big\{j<m:f(\{j,m\})=\{(1,1)\}\big\} \qquad \text{and} \qquad
			F_1 := \big\{j<m:f(\{j,m\})=\{(0,0)\}\big\}.
		\end{align*}
		For each $i \in\{0,1\}$, the set $F_i$ corresponds to the values of $j<m$ which are \emph{free}, that is, the values of $j$ for which $g(j)$ might not be the same for all functions $g \in F_i$.
		More precisely, our analysis shows that
		for $g_1,g_2 \in A_i$ and $j \notin F_i$ we have $g_1(j)=g_2(j)$, for $i\in \{0,1\}$.
		
		Let $G[F_0]:=\left(F_0,\restr{f}{\binom{F_0}{2}}\right)$ and  $G[F_1]:=\left(F_1,\restr{f}{\binom{F_1}{2}}\right)$ be the \emph{induced CSP} of $G$ on $F_0$ and $F_1$, respectively.
		It follows that
		\begin{align}\label{disjointsubsets}
			|A(G)|\le |A(G[F_0])|+|A(G[F_1])|.
		\end{align} 
		As $F_0$ and $F_1$ are subsets of $[m-1]$, by the induction hypothesis we have
		\begin{align}
			\label{inducedCSP}
			\big|A\big(G[F_0]\big)\big|\leq |F_0|+1 \quad \text{and} \quad \big|A\big(G[F_1]\big)\big|\leq |F_1|+1.
		\end{align}
		As $F_0$ and $F_1$ are disjoint, we have $|F_0|+|F_1|\le m-1$.
		Combining \eqref{disjointsubsets} and \eqref{inducedCSP}, we obtain 
		\begin{align*}
			|A(G)|\leq |F_0|+1+|F_1|+1\leq m+1.
		\end{align*}
		This completes our proof of Lemma~\ref{CSPcount}.
	\end{proof}
	
	Lemma~\ref{CSPcount} will be used in the proof of Lemma~\ref{Linkgraph} below. %We now proceed to the proof of Theorem~\ref{maincounting}.
	For the rest of this section, we fix natural numbers $k > r$ and a $3$-good list $L \se \big\{0,1,\ldots, \binom{k}{r}\big\}$.
	For simplicity, denote $f(n):=f(n,r,k,L)$ and let $\F(n)$ be the family of $(L,k)$-free $r$-graphs on $[n]$.
	In order to bound $d(n)$, we introduce other related quantities.
	For a set $A \se [n]$ and an $r$-graph $H \in \F(n)$, define
	\begin{align*}
		D(A,H,n):= \Big\{G \in \F(n): A \se e \ \ \, \forall \, e \in E(G) \Delta E(H) \Big\},
	\end{align*}
	where $\triangle$ denotes the symmetric difference.
	It is helpful to think on $D(A,H,n)$ as the set of $(L,k)$-free graphs on $[n]$ for which the edges and non-edges agree with $H$ on the $r$-sets not containing $A$.
	When $|A| = 1$, $D(A,H,n)$ is the set of $r$-graphs which extend $ \restr{H}{[n]\setminus A}$ to an $(L,k)$-free $r$-graph on $[n]$.
	It is also convenient to define
	\begin{align*}
		d(a,n):= \max \big \{|D(A,H,n)|: H \in \F(n) \big\},
	\end{align*}
	where $A$ is any set of size $a$. As the maximum is taken over all $H \in \F(n)$ and the set $\F(n)$ is closed for permuting the vertices of the graphs, it does not depend on the set $A$.
	When $a = 1$, we can think of $d(1,n)$ as the maximum number of extensions that an $r$-graph may have.
	We can easily see that 
	\begin{align}\label{recursionf}
		f(n) \le d(1,n)\cdot f(n-1)
	\end{align}
	for all $n \in \mathbb{N}$.
	Therefore, to bound $f(n)$ we will bound $d(1,n)$ for all $n \in \mathbb{N}$.
	
	%In order to bound $d(n)$, we introduce other related quantities.
	%Let $n, r$ and $i$ be natural numbers such that $n,r \ge i$ and let $H$ be an $r$-graph on $[n]$.
	%Define $D(i,n,H)$ to be the set of $(L,k)$-free $r$-graphs $G$ on $[n]$ for which the following holds.
	%If $e\in E(H)\triangle E(G)$, then $[i]\subseteq e$.
	%For $e \in \binom{[n]}{r}$ which does not contain $[i]$, we have $e \in E(G)$ if and only if $e \in E(H)$.
%	That is, for an $r$-graph $G \in D(i,n,H)$, the edges and non-edges of $G$ and $H$ can only differ in the $r$-sets containing $[i]$.
	%When $i \le n < r$, note that $|D(i,n,H)| = 1$, as the unique $r$-graph on $n$ vertices is the one with no edges.
	%Similarly, we define
	%Note that $D(1,n,H)=D(n,H)$ and hence $d(1,n) = d(n)$.
	
	In order to upper bound $d(1,n)$, we first provide an upper bound on $d(a,n)$ which depends on $d(a+1,v)$, for $v \in \{r,\ldots,n\}$.
	
	\begin{lemma}
		\label{recursionfora}
		Let $n, r, k$ and $a$ be natural numbers such that $k>r$ and $n > r-1 > a$.
		Then, we have 
		\begin{align*}
			d(a,n) \leq  \prod_{v=r}^n d(a+1,v).
		\end{align*}
	\end{lemma}
	\begin{proof}
		%We first bound $d(i,n)$ in terms of $d(i+1,n)$ and $d(i,n-1)$.
		Let $H$ be an $(L,k)$-free $r$-graph on $[n]$ and let $A \se [n-1]$ be a set of size $a$. 
		For an $r$-graph $G$ on $[n]$, denote by $G_{n-1}$ the $r$-graph on $[n-1]$ induced by $G$.
		Observe that if $G \in D(A,H,n)$, then $G_{n-1} \in D(A, H_{n-1}, n-1)$.
		We then partition the set $D(A,H,n)$ accordingly to the $r$-graph induced on the vertex set $[n-1]$. 
		For each $r$-graph  $J \in D(A,H_{n-1},n-1)$, let $T_{H}(A,J,n)$ be the set of $r$-graphs $G$ in $D(A,H,n)$ for which $G_{n-1}=J$. Then,
		\begin{align}
		\label{ineq:DinH}
		 D(A,H,n)=\bigcup_{J}T_H(A,J,n),
		\end{align}
		where the union is over $J \in D(A,H_{n-1},n-1)$.
		
		We claim that $T_H(A,J,n) \se D\big(\{n\}\cup A, G, n\big)$ for all $G \in T_H(A,J,n)$.
		To prove this, let $G$ and $G'$ be $r$-graphs in $T_H(A,J,n)$. As $G_{n-1} = J = G'_{n-1}$, we have $n \in e$ for all $e \in E(H') \Delta E(G)$.
		Moreover, as $G$, $G' \in D(A,H,n)$, we have $A \se e$ for all $e \in E(G) \Delta E(G')$.
		This implies that $\{n\} \cup A \se e$ for all $e \in E(G) \Delta E(G')$, which proves our claim.
		As $|A \cup \{n\}| = a+1$, we obtain
		%Observe that for a graph $G \in T_{H}(i,n,J)$ we have all the information about $r$-sets which do not contain $[i]$ and $r$-sets which do not contain $\{n\}$.
		%That is, for each $r$-graph $G \in T_H(i,n,J)$ and each $e \in \binom{[n]}{r}$ which does not contain $[i] \cup \{n\}$,  we have $e \in E(G)$ if and only if $e \in J \cup H$.The set $T_H(i,n,J)$ has a structure similar to $D(i+1,n,H\cup J)$. For any $r$-graph in $T_{H}(i,n,J)$, we have information about $r$-sets which do not contain $[i]\cup \{n\}$, while for $r$-graphs in $D(i+1,n,H\cup J)$ we have information about $r$-sets which do not contain $[i+1]$.From this, it follows that the maximum of $|T_H(i,n,J)|$ over $(L,k)$-free $r$-graphs $H$ and $J \in D(i,n-1,H_{n-1})$ is equal to $d(i+1,n)$.
		\begin{align}\label{ineq:Tandd}
			|T_{H}(A,J,n)|\le 
			d(a+1,n)
		\end{align}
		for all $H$ and $J \in D(A,H_{n-1},n-1)$. Combining \eqref{ineq:DinH} and \eqref{ineq:Tandd}, we get 
		\begin{align}
		\label{ineq:321}
		    |D(A,H,n)|\le |D(A,H_{n-1},n-1)| \cdot d(a+1,n) \le d(a,n-1) \cdot d(a+1,n).
		\end{align}
		As \eqref{ineq:321} holds for any $(L,k)$-free $r$-graph $H$ on $n$ vertices, \eqref{ineq:321} implies that
		\begin{align*}
			\dfrac{d(a,n)}{d(a,n-1)} \le d(a+1,n).
		\end{align*}
		Let $m\ge r$ be an integer. Performing the telescopic product with $n$ going from $r$ to $m$, we have
		\begin{align*}
			\dfrac{d(a,m)}{d(a,r-1)} \le \prod_{v=r}^m d(a+1,v).
		\end{align*}
		As $d(a,r-1)=1$, this completes our proof.	
	\end{proof}
	
	Recall that we need to bound $d(1,n)$ in order to bound $f(n)$. To do so, we use `backward' induction. 
	If we bound $d(a+1,n)$ for every $n \ge k$, then Lemma~\ref{recursionfora} yields a simple bound on $d(a,n)$ in terms of $d(a+1,n)$.
	Our next lemma concerns the base case, where we bound $d(r-1,n)$.
	
	\begin{lemma}
		\label{Linkgraph}
		Let $n, r$ and $k$ be natural numbers such that $n \ge r \ge 2$ and $k>r$.
		If $L\subseteq\left \{0,1,\ldots,\binom{k}{r} \right \}$ is a $3$-good list, then
		\begin{align*}
			d(r-1,n) \le 2^{k}n.
		\end{align*}
	\end{lemma}
	\begin{proof}
		Let $H$ be an $(L,k)$-free $r$-graph. Recall that $D\big([r-1],H,n\big)$ is the set of $(L,k)$-free $r$-graphs $G$ on $[n]$ for which the following holds. If $e\in E(H)\triangle E(G)$, then $[r-1]\subseteq e$. We associate the problem of counting the $r$-graphs in $D\big([r-1],H,n\big)$ to the problem of counting the $2$-colorings of the vertices in $[n]\setminus [r-1]$ under certain restrictions.
		The first step is to define for each $G \in D\big([r-1],H,n\big)$ a coloring $C_{G}:[n]\setminus [r-1]\to \{0,1\}$ as follows:
		\begin{align*}
			\begin{cases}
				C_{G}(i) = 1, \, \text{ if } \, [r-1]\cup \{i\} \in E(G), \\
				C_{G}(i) = 0, \, \text{ otherwise.}
			\end{cases}
		\end{align*}
		Observe that the number of $r$-graphs in $D\big([r-1],H,n\big)$ is equal to the number of colorings $c:[n]\setminus [r-1] \to \{0,1\}$ for which 
		\begin{align*}\left(H \cup \big\{i \cup [r-1]: i \in c^{-1}(1)\big\}\right) \setminus \big\{i \cup [r-1]: i \in c^{-1}(0)\big\}
		\end{align*}forms an $(L,k)$-free $r$-graph. As $|D\big([r-1],H,n\big)| \le 2^k$ if $n \le k$, from now on we assume that $n > k$.	
		We partition the set $D\big([r-1],H,n\big)$ according to a partial coloring on the set $S:=[k-2]\setminus[r-1]$. 
		For each coloring $c:S \to \{0,1\}$, let $D(c)$ be the set of $r$-graphs $G \in D\big([r-1],H,n\big)$ for which $\restr{C_{G}}{S}=c$. Now,
		\begin{align*}
		    D\big([r-1],H,n\big)=\bigcup_{c:S \to \{0,1\}}D(c).  
		\end{align*}
		We may think that all the edges and non-edges corresponding to $r$-sets of the form $[r-1]\cup \{i\}$, for $i \in S$, are fixed inside $D(c)$. 
		Our objective now is to bound the size of $D(c)$ via a constraint satisfaction problem on $[n]\setminus [k-2]$.
		To do so, we need to introduce some notation.
		For each pair of distinct vertices $\{i,j\} \se [n]\setminus [k-2]$,
		let $R_{i,j}$ be the collection of all $r$-sets in $[k-2]\cup \{i,j\}$ which are different from $[r-1]\cup \{i\}$ and $[r-1]\cup \{j\}$. 
		Observe that all $r$-graphs in $D(c)$ coincide on the $r$-sets in $R_{i,j}$.
		More precisely, for $G_1, G_2 \in D(c)$ and $e \in R_{i,j}$, we have $e \in E(G_1)$ if and only if $e \in E(G_2)$.
		Let $\|R_{i,j}\|_c$ be the number of edges in $R_{i,j}$ which are in common for every $r$-graph in $D(c)$.
		Note that we cannot have $r$-graphs $G_1, \, G_2$ and $G_3$ in $D(c)$ which induce $\|R_{i,j}\|_c$, $\|R_{i,j}\|_c+1$ and $\|R_{i,j}\|_c+2$ edges in $[k-2]\cup \{i,j\}$, respectively, as this would contradict the assumption that $L \cap \{i-1,i,i+1\} \neq \emptyset$ for all $i \in \left [ \binom{k}{r}-1 \right]$.
		
		For each $\{i,j\} \se [n]\setminus [k-2]$, let $t^c_{i,j}\in \{0,1,2\}$ be such that $\|R_{i,j}\|_c+t^c_{i,j}\notin L$.
		Then, for every $G \in D(c)$ we must have $C_{G}(i)+C_{G}(j) \neq t_{i,j}^c$, otherwise $G$ would have a forbidden structure on $[k-2]\cup \{i,j\}$.
		If $t_{i,j}^c=0$ or $2$, this is equivalent to the constraint $\big(C_{G}(i),C_{G}(j)\big)\neq (t_{i,j}^c/2, t_{i,j}^c/2)$; and if 
		$t_{i,j}^c = 1$, it is equivalent to the constraint $\big(C_{G}(i),C_{G}(j)\big)\notin \{(0,1), \, (1,0)\}$.
		We can then define the following constraint function $f_c: \binom{[n]\setminus [k-2]}{2} \to \C$:
		\begin{align*}
			f(\{i,j\}):=\begin{cases}
				\{(0,0)\}, &\text{if } t_{i,j}^c=0,\\
				\{(1,0),(0,1) \}, &\text{if } t_{i,j}^c=1,\\
				\{(1,1) \}, &\text{if } t_{i,j}^c=2.
			\end{cases}
		\end{align*}
		Let $G_c = \big( [n]\setminus [k-2], f_c\big)$ be a CSP.
		It follows that $|D(c)| \le |A(G_c)|$, where $A(G_c)$ is the set of satisfying assignments of the CSP $G_c$.
		By Lemma~\ref{CSPcount}, we have $|A(G_c)| \le n-k+3$ and hence 
		\begin{align*}
			|D\big([r-1],H,n\big)|&\le \sum \limits_{c:\, S \to \{0,1\}} |D(c)|  \le n \cdot 2^{|S|}.
		\end{align*}
		As $|S| \le k$, this proves our lemma.
		
	\end{proof}
	
	Combining Lemmas~\ref{recursionfora} and~\ref{Linkgraph}, we obtain the following corollary.
	\begin{corollary}
		\label{recursionasolved}
		Let $n, k$ and $r$ be natural numbers such that $n,k \ge r \ge 2$ and let
		$L\subseteq\left \{0,1,\ldots,\binom{k}{r} \right \}$ be a $3$-good list.
		Then, for all $i \in [r-1]$ we have
		\begin{align*}
			d(r-i,n)\leq 2^{kn^{i-1} +n^{i-1}\log n}.
		\end{align*}
	\end{corollary}
	\begin{proof}
		We proceed by induction on $i$. By Lemma~\ref{Linkgraph}, the statement holds for $i=1$. Now, assume that the lemma holds for some  $i \in [r-2]$.
		By Lemma~\ref{recursionfora}, we obtain
		\begin{align*}
			%\label{inequalitya1}
			 d(r-(i+1),n) & \, \leq \,  \prod_{v=r}^n d(r-i,v)
			 \le \, \prod_{v=r}^n 2^{kv^{i-1} +v^{i-1}\log v}
			 \le \, 2^{kn^{i}+n^{i}\log n},
		\end{align*}
		as required.
	\end{proof}
	
	We are now ready to complete the proof of Theorem~\ref{maincounting}.
	Let $n \ge k>r\ge 2$ be integers and $L\subseteq \big\{0,1,\ldots, \binom{k}{r}\big\}$ be a $3$-good list. 
	By~\eqref{recursionf}, we have $f(n) \le d(n) \cdot f(n-1)$ and hence
	\begin{align*}
		f(n) \le f(k)\cdot \prod_{v=k+1}^{n} d(v).
	\end{align*}
	As $d(v) = d(1,v)$, by Corollary~\ref{recursionasolved} we have $d(v) \le 2^{kn^{r-2}+v^{r-2}\log v}$.
	From this together with the trivial bound $f(k) \le 2^{k^r}$, we obtain
	\begin{align*}
		f(n)\, & \le \,  2^{k^r} \prod_{v=k+1}^n 2^{kv^{r-2} +v^{r-2}\log v} \le \, 2^{k^r+kn^{r-1}+n^{r-1}\log n},
	\end{align*}
	as required.

	\section{Proof of Theorems~\ref{maincountingcorl} and~\ref{theo:tablevalues}}
	\label{sec:tableexplanation}
In this section we prove Theorem~\ref{theo:tablevalues}, that is, we determine $f(n,3,4,L)$ asymptotically for all possible $L$. 
In particular, we prove Theorem~\ref{maincountingcorl}.
For simplicity, we denote $f(n,L)=f(n,3,4,L)$ and assume that $n \ge 4$ throughout this section.

For a list $L \se \{0,1,2,3,4\}$, define $L^c = \{4-i: i \in L\}$.
Observe that $f(n,L)=f(n,L^c)$, as a $3$-graph $G$ does not induce $i$ edges on $4$ vertices if and only if its complement $G^c$ does not induce $4-i$ edges on $4$ vertices. 
In light of this, to prove of Theorem~\ref{theo:tablevalues} it is sufficient to bound $f(n,L)$ for only one set $L$ in each row of Table 1 below.
When $\{0,4\} \se L$, our proof trivially follows from Ramsey's theorem, see Claim~\ref{L04}.
When $\{2\} \se L\se \{0,1,2\}$, $L$ is $3$-good and hence the upper bound on $f(n,L)$ follows from Theorem~\ref{maincounting}. The lower bound is obtained via the same greedy procedure used in the proof of Lemma~\ref{listsharp1}, see Claim~\ref{L2} for more details.
These and the trivial cases already occupy a good proportion of the table.
The remaining cases are more delicate and we need to deal with each of them separately.
We start with the proof of Theorem~\ref{maincountingcorl}.

\begin{table}[h]
	\begin{center}
		\begin{tabular}{ |c|c|c| } 
			\hline
			$L$ & $f(n,3,4,L)$  &  proof   \\
			\hline
			$\emptyset$ & $2^{\binom{n}{3}}$ & trivial \\
			\hline
			\{0\},\ \{4\} & $2^{\Theta(n^3)}$ & Claim~\ref{L34}  \\
			\hline
			\{1\},\ \{3\} & $2^{\Theta(n^3)}$ &  follows from $\{3,4\}$  \\
			\hline
			\{0,1\},\ \{3,4\} & $2^{\Theta(n^3)}$ & Claim~\ref{L34}  \\
			\hline
			\{2\} & $2^{\Theta(n^2\log n)}$ & Claim~\ref{L2}  \\
			\hline
			\{0,2\},\ \{2,4\} & $2^{\Theta(n^2\log n)}$ & Claim~\ref{L2}  \\
			\hline
			\{0,3\},\ \{1,4\} & $2^{\Theta(n^2\log n)}$ & Theorem~\ref{maincountingcorl}  \\
			\hline
			\{1,2\},\ \{2,3\} & $2^{\Theta(n^2\log n)}$ & Claim~\ref{L2}  \\
			\hline
			\{0,1,2\},\ \{2,3,4\} & $2^{\Theta(n^2\log n)}$ & Claim~\ref{L2}  \\
			\hline
			\{1,3\} & $2^{\binom{n-1}{2}}$ & Lemma~\ref{L13}   \\
			\hline
			\{0,1,3\},\ \{1,3,4\} & $2^{\Theta(n\log n)}$ & Lemma~\ref{L013}  \\
			\hline
			\{0,2,3\},\ \{1,2,4\} & $n+1$ for $n\geq 5$ & Claim~\ref{L023} \\
			\hline
			\{1,2,3\} & 2 & Claim~\ref{L123} \\
			\hline
			\{1,2,3,4\},\ \{0,1,2,3\} & 1  & trivial  \\
			\hline
			\{0,4\} & $0$ for $n\geq 13$ & Claim~\ref{L04}  \\
			\hline
			\{0,1,4\},\ \{0,3,4\} & $0$ for $n\geq 13$ & Claim~\ref{L04}  \\
			\hline
			\{0,2,4\} & $0$ for $n\geq 13$ & Claim~\ref{L04}  \\
			\hline
			\{0,2,3,4\},\ \{0,1,2,4\} & $0$ for $n\geq 13$ & Claim~\ref{L04}  \\
			\hline
			\{0,1,3,4\} & $0$ for $n\geq 13$ & Claim~\ref{L04}  \\
			\hline
			\{0,1,2,3,4\} & 0 & trivial  \\
			\hline
		\end{tabular}
	\end{center}
	\caption{
		\label{tab:tableforL}
		This table shows the values for $f(n,3,4,L)$ for all possible $L$.}
\end{table}

	\begin{proof}[Proof of Theorem~\ref{maincountingcorl}]
		The upper bound follows immediately from Theorem~\ref{maincounting}.
		For the lower bound, we consider the family of $3$-graphs $Q(n)$ presented in the introduction.
		Let $[n]=V_1\cup V_2\cup V_3$ be a partition of $[n]$ with $\big | |V_i|-|V_j| \big |\le 1$ for all $i,j \in [3]$, which is the same partition used in the definition of $C_n$.
		Recall that a $3$-graph is in $Q(n)$ if it is obtained from $C_n$ by removing a linear $3$-graph with the additional property that every edge contains one element from each of the classes $V_1$, $V_2$ and $V_3$.
		As every $4$-set in $C_n$ has either $0$ or $3$ edges,
		for a $3$-graph $H \in Q(n)$ every $4$-set has either $0$ or $2$ edges.
		As $Q(n)$ is a family of $(\{1,4\},4)$-free $3$-graphs on $n$ vertices, to lower bound $f(n,\{1,4\})$ it suffices to lower bound $|Q(n)|$.
		
%		we first consider an extremal $(\{1,4\},4)$-free $3$-graph $C_n$ with vertex set $[n]$, see Figure~\ref{fig:Cn}.
%		This $3$-graph is constructed\footnote{This construction is due to Paul Turán.} as follows.
%		Let $(V_i)_{i \in [3]}$ be a partition of the vertex set $[n]$ such that $\big | |V_i|-|V_j| \big |\le 1$ for all $i,j \in [3]$.
%		An edge is placed in $C_n$ if it intersects each of the classes $V_1$, $V_2$ and $V_3$, or if for some $i \in [3]$ it contains two elements of $V_i$ and one of $V_{i+1}$, where the indices are understood modulo $3$.
%		We now construct a family of subgraphs of $C_n$ which are $(\{1,4\},4)$-free. 
%		To do so,
%		let $L(n)$ be the family of $3$-graphs $H$ on $[n]$ with the following properties: 
%		(a) every edge in $H$ contains one vertex from each of the classes $V_1$, $V_2$ and $V_3$; (b) two edges in $H$ intersect in at most one vertex.
%		Define 
%		\begin{align*}
%			Q(n) = \{C_n-H: H \in L(n)\},
%		\end{align*}
%		where $C_n-H$ denotes the subgraph with vertex set $V(C_n)$ and edge set $E(C_n)\setminus E(H)$.
%		A careful analysis shows that every $3$-graph in $Q(n)$ is $(\{1,4\},4)$-free, we skip the details.
%As $|Q(n)| = |L(n)|$, to lower bound $|Q(n)|$ we lower bound $|L(n)|$ via the following greedy procedure, which is similar to the one in the proof of Lemma~\ref{listsharp3}.
		Let $L(n)$ be the family of linear $3$-graphs on $[n]$ with the additional property that every edge intersects each of the classes $V_1,V_2$ and $V_3$. 
		Clearly, we have $|Q(n)|=|L(n)|$.
		Now, we lower bound $|L(n)|$ via the following greedy procedure, which is similar to the one in the proof of Lemma~\ref{listsharp3}.
		Let $e_1$ be a $3$-set in $[n]$ such that $|e_1 \cap V_j|=1$ for all $j \in [3]$.
		For $i>1$, let $e_i$ be a $3$-set in $[n]$ such that $|e_1 \cap V_j|=1$ for all $j \in [3]$ and such that $|e_i \cap e_k| \le 1$ for all $k \in [i-1]$.
		When an edge $e_i$ with this property cannot be found, the procedure stops and output $\{e_1,\ldots,e_{i-1}\}$.
		Observe that any $3$-graph obtained from this procedure belongs to $L(n)$.
		As there are at most $ni$ sets of size $3$ which intersect some element of $\{e_1,\ldots, e_{i-1}\}$ in $2$ vertices or more, we have at least $|V_1||V_2||V_3|-ni$ choices for $e_i$.
		This implies that the procedure lasts for at least $\frac{n^{2}}{27}-\frac{n^2}{\log n}$ steps.
		Moreover, as we have $|V_1||V_2||V_3|-ni \ge \frac{n^3}{2\log n}$ for all $i \le \frac{n^{2}}{27}-\frac{n^2}{\log n}$, the number of $3$-graphs in $L(n)$ and hence in $Q(n)$ is at least
		\begin{align}\label{sizeofQ}
		|Q(n)|\geq
			\frac{\left(\frac{n^3}{2\log n}\right)^{\frac{n^{2}}{27}-\frac{n^2}{\log n}}}{\left(\left \lfloor\frac{n^{2}}{27}-\frac{n^2}{\log n}\right \rfloor\right)!} \geq \left( \frac{n}{\log n} \right)^{\frac{n^{2}}{27}-\frac{n^2}{\log n}} =2^{ \big(\frac{1}{27}+o(1)\big)n^2\log n}.
		\end{align}
		The factorial term above takes the double counting into consideration.
		Combining this bound with the upper bound from Theorem~\ref{maincounting}, we obtain $f(n,\{1,4\})=2^{\Theta(n^{2}\log n)}$.
			\end{proof}
	
	\begin{claim}
		\label{L04}
		If $n\geq 13$ and $\{0,4\}\subseteq L$, then $f(n,L)=0$.
	\end{claim}
	\begin{proof}
		Let $R_3(4,4)$ be the smallest integer $n$ such that any red and blue edge-coloring of the complete $3$-graph $K_n^3$ contains a red copy of $K_4^3$ or a blue copy of $K_4^3$.
		The hypergraph Ramsey number $R_3(4,4)$ was determined by McKay and Radziszowski in~\cite{MR1095846}, where they showed that $R_3(4,4)=13$.
		We conclude that for all sets $L$ such that $\{0,4\} \se L$ there is no $(L,4)$-free $3$-graph on at least $13$ vertices. 
	\end{proof}

	\begin{claim}
		\label{L2}
		Let $L$ be a list such that $\{2\} \se L \se \{0,1,2\}$.
		Then, $f(n,L)=2^{\Theta(n^2\log n)}$.
	\end{claim}
	\begin{proof}
		Let $L$ be a list such that $\{2\}\se L \se \{0,1,2\}$.
		Then, for every $i \in [3]$ we have $L \cap \{i-1,i,i+1\} \neq \emptyset$ and hence $L$ is $3$-good.
		By Theorem~\ref{maincounting}, it follows that $f(n,L) \le 2^{\Theta(n^2\log n)}$.
		
		To show a lower bound of the same order, we consider the set $M(n,3)$ of $3$-graphs on $[n]$ such that every pair of vertices is contained in at most one edge.
		For a graph $H \in M(n,3)$, we have that the graph $K_n^3-H$ is $(L,4)$-free and hence $f(n,L)\ge |M(n,3)|$.
		A lower bound on $|M(n,3)|$ was already obtained in the proof of Lemma~\ref{listsharp1}, where we showed that $|M(n,3)|\ge 2^{\Theta(n^2\log n)}$.
	\end{proof}

	To prove the next two lemmas, it is convenient to define the link graph of a vertex.
	For a $3$-graph $H$ and a vertex $v$ of $H$, we define $S_{H}(v)$ to be the graph with vertex set $V(H)\setminus\{v\}$ and edge set 
	\begin{align*}
		E\big(S_{H}(v)\big):=\{e\setminus\{v\}: v \in e, e \in E(H)\}.
	\end{align*}
	We refer to $S_{H}(v)$ as the \emph{link graph} of $v$ in $H$.
	\begin{lemma}
		\label{L013}
		If $L=\{0,1,3\}$, then $f(n,L)=2^{\Theta(n \log n)}$.
	\end{lemma}
	\begin{proof}
		Let $\F(n,L)$ be the family of $(L,4)$-free $3$-graphs on $[n]$ and $\S:=\{S_G(n): \, G \in \F(n,L) \big \}$ be a family of link graphs. 
		We claim that there is a bijection between $\F(n,L)$ and $\S$.
		To show this, take an arbitrary graph $A \in \S$ and let $G$ be such that $S_G(n) = A$.
		For a $3$-set  $\{a_1, a_2, a_3\}$ in $[n-1]$, the following holds:
		\begin{enumerate}
			\item[$(1)$] $A$ has at least one edge in $\{a_1, a_2, a_3\}$. Indeed, otherwise the $3$-graph $G$ induces at most one edge in $\{n, a_1, a_2, a_3\}$, which is a contradiction. 
			
			\item[$(2)$] If $A$ has $1$ or $3$ edges in $\{a_1, a_2, a_3\}$, then $a_1a_2 a_3 \in E(G)$. Otherwise, we have a forbidden structure in $\{n,a_1, a_2, a_3\}$.
			\item[$(3)$] If $A$ has $2$ edges in $\{a_1, a_2, a_3\}$, then $a_1a_2 a_3 \notin E(G)$. Otherwise, we have a forbidden structure in $\{n,a_1, a_2, a_3\}$.
		\end{enumerate}
		From items (1)--(3) it follows that for each $A \in \S$ there is an unique $G \in \F(n,L)$ such that $S_G(n)=A$. From now on, we denote this $3$-graph by $G_A$.

		As there is a bijection between $\F(n,L)$ and $\S$, to bound $|\F(n,L)|$ it suffices to  determine all graphs which belong to $\S$.
		Fix some $A \in \S$ and let $A^c$ be its complement, that is, the graph with vertex set $[n-1]$ and edge set $\binom{[n-1]}{2}\setminus E(A)$. 
		By item (1), we already know that $A^c$ must be triangle-free.
		To see which other conditions $A$ must satisfy, we analyze the graph induced by $A$ in each $4$-set in $[n-1]$.
		We first observe that graphs on $4$ vertices can be divided into $3$ categories:
		($i$) the complement contains a triangle; ($ii$) the graph induces a $C_4$;
		($iii$) the items ($i$) and ($ii$) do not hold. 
		
		\begin{center}
			\begin{tikzpicture}[scale=1]
				\clip(-3.5,-1.6064884705671694) rectangle (11,0.42553079548812683);
				\draw [line width=0.8pt] (-3,0)-- (-3,-1);
				\draw [line width=0.8pt] (-2,0)-- (-2,-1);
				
				\draw [line width=0.8pt] (-0.5,0)-- (-0.5,-1);
				\draw [line width=0.8pt] (-0.5,-1)-- (0.5,0);
				\draw [line width=0.8pt] (0.5,0)-- (-0.5,0);
				
				\draw [line width=0.8pt] (2,-1)-- (2,0);
				\draw [line width=0.8pt] (2,0)-- (3,0);
				\draw [line width=0.8pt] (3,0)-- (3,-1);
				
				\draw [line width=0.8pt] (4.5,-1)-- (4.5,0);
				\draw [line width=0.8pt] (4.5,0)-- (5.5,0);
				\draw [line width=0.8pt] (4.5,-1)-- (5.5,0);
				\draw [line width=0.8pt] (4.5,-1)-- (5.5,-1);
				
				\draw [line width=0.8pt] (7,0)-- (7,-1);
				\draw [line width=0.8pt] (7,-1)-- (8,0);
				\draw [line width=0.8pt] (8,0)-- (7,0);
				\draw [line width=0.8pt] (7,0)-- (8,-1);
				%\draw [line width=0.8pt] (8,-1)-- (7,-1);
				\draw [line width=0.8pt] (8,-1)-- (8,0);
				
				\draw [line width=0.8pt] (9.5,0)-- (9.5,-1);
				\draw [line width=0.8pt] (9.5,-1)-- (10.5,0);
				\draw [line width=0.8pt] (10.5,0)-- (9.5,0);
				\draw [line width=0.8pt] (9.5,0)-- (10.5,-1);
				\draw [line width=0.8pt] (9.5,-1)-- (10.5,-1);
				\draw [line width=0.8pt] (10.5,-1)-- (10.5,0);
				\begin{scriptsize}
					\draw [fill=black] (-3,0) circle (1pt);
					\draw [fill=black] (-3,-1) circle (1pt);
					\draw [fill=black] (-2,0) circle (1pt);
					\draw [fill=black] (-2,-1) circle (1pt);
					
					\draw [fill=black] (-0.5,0) circle (1pt);
					\draw [fill=black] (-0.5,-1) circle (1pt);
					\draw [fill=black] (0.5,0) circle (1pt);
					\draw [fill=black] (0.5,-1) circle (1pt);
					
					\draw [fill=black] (2,0) circle (1pt);
					\draw [fill=black] (2,-1) circle (1pt);
					\draw [fill=black] (3,-1) circle (1pt);
					\draw [fill=black] (3,0) circle (1pt);
					
					\draw [fill=black] (4.5,0) circle (1pt);
					\draw [fill=black] (4.5,-1) circle (1pt);
					\draw [fill=black] (5.5,-1) circle (1pt);
					\draw [fill=black] (5.5,0) circle (1pt);
					
					\draw [fill=black] (7,0) circle (1pt);
					\draw [fill=black] (7,-1) circle (1pt);
					\draw [fill=black] (8,0) circle (1pt);
					\draw [fill=black] (8,-1) circle (1pt);
					
					\draw [fill=black] (9.5,0) circle (1pt);
					\draw [fill=black] (9.5,-1) circle (1pt);
					\draw [fill=black] (10.5,0) circle (1pt);
					\draw [fill=black] (10.5,-1) circle (1pt);
					
				\end{scriptsize}
			\end{tikzpicture}

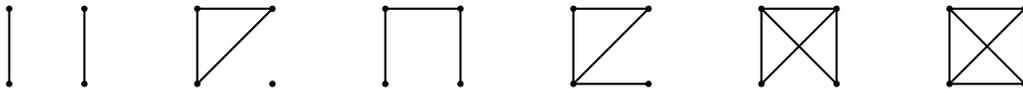
		Figure 2: All non-isomorphic graphs on 4 vertices satisfying ($iii$).
		\end{center}
		Let $\{a,b,c,d\}$ be a set of size $4$ in $[n-1]$.
		As $A^c$ is triangle-free, we already know that $\{a,b,c,d\}$ does not satisfy item ($i$) in $A$.
		Now, we claim that $A$ does not induce a $C_4$ in $\{a,b,c,d\}$, hence item ($ii$) does not hold.
		Indeed, if $A$ induces a $C_4$ in $\{a,b,c,d\}$, then it follows from item (3) that $G_A$ has no edge in $\{a,b,c,d\}$, which is a contradiction.
		We conclude that every set of $4$ vertices in $A$ satisfies item $(iii)$. This is equivalent to saying that $A^c$ is free of triangles and free of induced matchings of size $2$.

		For $m \in \mathbb{N}$, let $\F_{\triangle,M}(m)$ be the family of graphs on $[m]$ which are free of triangles and induced matchings of size $2$.
		We have seen that if $A \in \S$, then $A^c \in \F_{\triangle,M}(n-1)$.
		Now, we claim that the converse also holds.
		Let $H$ be a graph such that $H^c \in \F_{\triangle,M}(n-1)$ and denote by $G^H$ the $3$-graph on $[n]$ which satisfies items (1)--(3), with $A$ replaced by $H$ and $G$ replaced by $G^H$. By the definition of $G^H$, all $4$-sets in $[n]$ containing $n$ do not induce a forbidden structure in $G^H$.
		Now, let $\{a,b,c,d\}$ be an arbitrary $4$-set in $[n-1]$.
		As $H^c \in \F_{\triangle,M}(n-1)$, the graph induced by $H$ in $\{a,b,c,d\}$ satisfies item ($iii$).
		We represent in Figure 2 all non-isomorphic graphs that $H$ can induce on $\{a,b,c,d\}$.
		Using items (2) and (3), a careful analysis on the number of edges in $\{a,b,c,d\}$ shows that $G^H$ does not induce any forbidden structure. Therefore, we have $G^H \in \F(n,L)$, which implies that $H \in \S$.
		In particular, $G^H = G_H$.

		Now it remains to bound the size of $\F_{\triangle,M}(m)$.
		To do so, we first claim that a graph $G \in \F_{\triangle,M}(m)$ has chromatic number at most $3$.
		Indeed, fix $G \in \F_{\triangle,M}(m)$ and let $uv$ be any edge of $G$.
		Let $N(u)$ and $N(v)$ be the neighborhoods of $u$ and $v$, respectively.
		These neighborhoods cannot intersect, otherwise we create a triangle. Moreover, the set $[m]\setminus \big ( \{u,v\} \cup N(u)\cup N(v) \big)$ cannot have an edge, otherwise we create an induced matching of size $2$ with $uv$.
		It follows that we can divide the graph into three disjoint independent sets: $A_1=\{u\}\cup N(v)$, $A_2 = \{v\} \cup N(u)$ and $A_3:= [m]\setminus \big ( \{u,v\} \cup N(u)\cup N(v) \big)$.
		This proves our claim.
		
		Let $b(m)$ be the number of bipartite graphs with $m$ vertices in each class and with no induced matching of size $2$. As every graph in $\F_{\triangle,M}(m)$ has chromatic number at most $3$ and does not induce a matching of size $2$, we have
		\begin{align}\label{boundfamF}
			b(\lfloor m/2 \rfloor ) \le |\F_{\triangle,M}(m)| \le 3^m \cdot (b( m ))^3.
		\end{align}
		The factor of $3^m$ in the upper bound accounts for the number of ways to partition the set $[m]$ into $3$ parts.
		To bound $b(m)$, we use an argument which appears in~\cite{LZ97}.
		In~\cite{LZ97}, the authors observed that a bipartite graph with parts $A$ and $B$ has no induced matching if and only if for every $a_1,a_2 \in A$ we have $N(a_1) \se N(a_2)$ or $N(a_2) \se N(a_1)$. That is, the set $\{N(a):a \in A\}$ forms a chain.
		Observe that the number of chains of length $m$ is equal to the number of 
		ways to distribute the elements of $[m]$ into $m$ disjoint labeled sets $S_1,\ldots,S_m$. On one hand, this number is at least $m! = m^{\Theta(m)}$, which is the number of ways to place exactly one element in each $S_i$.
		On the other hand, we have the trivial upper bound $m^m$,
		hence the number of chains of length $m$ is of order $m^{\Theta(m)}$.
		By~\eqref{boundfamF}, 
		we obtain $ |\F_{\triangle,M}(m)| = m^{\Theta(m)}$ and therefore
		\begin{align*}
			|\F(n,3,4,L)| = |\F_{\triangle,M}(n-1)| = n^{\Theta(n)}.
		\end{align*}
	\end{proof}
	
	\begin{lemma}
		\label{L13}
		If $L=\{1,3\}$, then $f(n,L)=2^{\binom{n-1}{2}}$.
	\end{lemma}
	\begin{proof}
		
		Let $\F(n,L)$ be the family of $(L,4)$-free $3$-graphs on $[n]$ and $\S:=\{S_G(n): \, G \in \F(n,L) \big \}$ be a family of link graphs. 
		We claim that there is a bijection between $\F(n,L)$ and $\S$.
		To show this, we proceed as in the proof of Lemma~\ref{L013}.
		Take an arbitrary graph $A \in \S$ and let $G$ be such that $S_G(n) = A$.
		For a $3$-set $\{a_1, a_2, a_3\}$ in $[n-1]$, the following holds:
		\begin{enumerate}
			\item[$(1)$] If $A$ has $0$ or $2$ edges in $\{a_1, a_2, a_3\}$, then
			$a_1a_2 a_3 \notin E(G)$, as otherwise we have a forbidden structure in $\{a_1, a_2, a_3,n\}$.
			\item[$(2)$] If $A$ has $1$ or $3$ edges in $\{a_1, a_2, a_3\}$, then
			$a_1a_2 a_3\in E(G)$, as otherwise we create a forbidden structure in $\{n,a_1, a_2, a_3\}$.
		\end{enumerate}
		From items (1) and (2) it follows that for each $A \in \S$ there is an unique $G \in \F(n,L)$ such that $S_G(n)=A$.

		As there is a bijection between $\F(n,L)$ and $\S$, to bound $|\F(n,L)|$, it suffices to  determine all graphs which belong to $\S$.
		We claim that $\S$ contains all graphs on $[n-1]$.
		Let $H$ be a graph on $[n-1]$ and denote by $G^H$ the $3$-graph on $[n]$ which satisfies items (1) and (2), with $A$ replaced by $H$ and and $G$ replaced by $G^H$. 
		By the definition of $G^H$, 
		every $4$-set in $[n]$ containing $n$ does not induce a forbidden structure in $G^H$.
		Now, let $\{a,b,c,d\}$ be an arbitrary $4$-set in $[n-1]$.
		In Figure 3 below, we show all possible non-isomorphic graphs on $\{a, b, c, d\}$ and their associated $3$-graphs satisfying items (1) and (2).
		We can see that for any graph induced by $H$ on $\{a,b,c,d\}$, the $3$-graph $G^H$ does not induce any forbidden structure on $\{a,b,c,d\}$. 
		As the number of graphs on $[n-1]$ is $2^{\binom{n-1}{2}}$, we have $|\F(n,L)| = 2^{\binom{n-1}{2}}$.
		
		\begin{center}
			\begin{tikzpicture}
				%\clip(-0.2,-0.2) rectangle (17,2);
				[
				he/.style={draw, semithick},        % he = hyper edge
				ce/.style={draw, dashed, semithick}, % ce = condition edge
				]
				% 0 edges
				\draw [fill=black] (0,0) circle (1pt);
				\draw [fill=black] (1,0) circle (1pt);
				\draw [fill=black] (0,1) circle (1pt);
				\draw [fill=black] (1,1) circle (1pt);
				
				\node (1) at (0,0) {};
				\node (2) at (1,0) {};
				\node (3) at (0,1) {};
				\node (4) at (1,1) {};
				
				% emptyhypergraph
				
				% 1 edge
				\draw [fill=black] (2.5,0) circle (1pt);
				\draw [fill=black] (3.5,0) circle (1pt);
				\draw [fill=black] (2.5,1) circle (1pt);
				\draw [fill=black] (3.5,1) circle (1pt);
				
				\node (5) at (2.5,0) {};
				\node (6) at (3.5,0) {};
				\node (7) at (2.5,1) {};
				\node (8) at (3.5,1) {};
				
				\draw [line width=0.8pt] (2.5,0)-- (3.5,0);

				% 2 edges
				\draw [fill=black] (5,0) circle (1pt);
				\draw [fill=black] (6,0) circle (1pt);
				\draw [fill=black] (5,1) circle (1pt);
				\draw [fill=black] (6,1) circle (1pt);
				
				\node (9) at (5,0) {};
				\node (10) at (6,0) {};
				\node (11) at (5,1) {};
				\node (12) at (6,1) {};
				
				\draw [line width=0.8pt] (5,0)-- (5,1);
				\draw [line width=0.8pt] (6,0)-- (6,1);

				\draw [fill=black] (7.5,0) circle (1pt);
				\draw [fill=black] (8.5,0) circle (1pt);
				\draw [fill=black] (7.5,1) circle (1pt);
				\draw [fill=black] (8.5,1) circle (1pt);
				
				\node (13) at (7.5,0) {};
				\node (14) at (8.5,0) {};
				\node (15) at (7.5,1) {};
				\node (16) at (8.5,1) {};
				
				\draw [line width=0.8pt] (7.5,0)-- (7.5,1);
				\draw [line width=0.8pt] (7.5,0)-- (8.5,0);
				
				% 3 edges
				\draw [fill=black] (10,0) circle (1pt);
				\draw [fill=black] (11,0) circle (1pt);
				\draw [fill=black] (10,1) circle (1pt);
				\draw [fill=black] (11,1) circle (1pt);
				
				\node (17) at (10,0) {};
				\node (18) at (11,0) {};
				\node (19) at (10,1) {};
				\node (20) at (11,1) {};
				
				\draw [line width=0.8pt] (10,0)-- (11,0);
				\draw [line width=0.8pt] (10,0)-- (10,1);
				\draw [line width=0.8pt] (10,0)-- (11,1);

				\draw [fill=black] (12.5,0) circle (1pt);
				\draw [fill=black] (13.5,0) circle (1pt);
				\draw [fill=black] (12.5,1) circle (1pt);
				\draw [fill=black] (13.5,1) circle (1pt);
				
				\node (21) at (12.5,0) {};
				\node (22) at (13.5,0) {};
				\node (23) at (12.5,1) {};
				\node (24) at (13.5,1) {};
				
				\draw [line width=0.8pt] (12.5,0)-- (13.5,0);
				\draw [line width=0.8pt] (12.5,0)-- (12.5,1);
				\draw [line width=0.8pt] (12.5,1)-- (13.5,1);

%				
%				\draw [fill=black] (15,0) circle (1pt);
%				\draw [fill=black] (16,0) circle (1pt);
%				\draw [fill=black] (15,1) circle (1pt);
%				\draw [fill=black] (16,1) circle (1pt);
%				
%				\node (25) at (15,0) {};
%				\node (26) at (16,0) {};
%				\node (27) at (15,1) {};
%				\node (28) at (16,1) {};
%				
%				\draw [line width=0.8pt] (15,0)-- (16,0);
%				\draw [line width=0.8pt] (15,0)-- (15,1);
%				\draw [line width=0.8pt] (15,1)-- (16,0);
				
				\begin{pgfonlayer}{bg}
					% 1 edge
					\draw[thick, fill = red, fill opacity = .25] \hedgeiii{5}{8}{6}{3mm};
					\draw[thick, fill = blue, fill opacity = .25] \hedgeiii{5}{7}{6}{3mm};
					
					%2 edges
					\draw[thick, fill = blue, fill opacity = .25] \hedgeiii{11}{10}{9}{3mm};
					\draw[thick, fill = red, fill opacity = .25] \hedgeiii{12}{10}{9}{3mm};
					\draw[thick, fill = green, fill opacity = .25] \hedgeiii{11}{12}{9}{3mm};
					\draw[thick, fill = yellow, fill opacity = .25] \hedgeiii{11}{12}{10}{3mm};
					
					\draw[thick, fill = blue, fill opacity = .25] \hedgeiii{13}{15}{16}{3mm};
					\draw[thick, fill = red, fill opacity = .25] \hedgeiii{13}{16}{14}{3mm};

					%3 edges
					\draw[thick, fill = blue, fill opacity = .25] \hedgeiii{24}{22}{21}{3mm};
					\draw[thick, fill = red, fill opacity = .25] \hedgeiii{23}{24}{22}{3mm};
					
					%\draw[thick, fill = blue, fill opacity = .25] \hedgeiii{27}{26}{25}{3mm};
					%\draw[thick, fill = red, fill opacity = .25] \hedgeiii{28}{26}{25}{3mm};
					%\draw[thick, fill = green, fill opacity = .25] \hedgeiii{27}{28}{25}{3mm};
					%\draw[thick, fill = yellow, fill opacity = .25] \hedgeiii{27}{28}{26}{3mm};
					
				\end{pgfonlayer}
				
			\end{tikzpicture}
		\end{center}
		
		\begin{center}
			\begin{tikzpicture}
				%\clip(-0.2,-0.2) rectangle (17,2);
				[
				he/.style={draw, semithick},        % he = hyper edge
				ce/.style={draw, dashed, semithick}, % ce = condition edge
				]
				% 3 edges
				\draw [fill=black] (0,0) circle (1pt);
				\draw [fill=black] (1,0) circle (1pt);
				\draw [fill=black] (0,1) circle (1pt);
				\draw [fill=black] (1,1) circle (1pt);
				
				\node (1) at (0,0) {};
				\node (2) at (1,0) {};
				\node (3) at (0,1) {};
				\node (4) at (1,1) {};
				
				\draw [line width=0.8pt] (0,0)-- (1,0);
				\draw [line width=0.8pt] (0,0)-- (0,1);
				\draw [line width=0.8pt] (0,1)-- (1,0);
				
				% emptyhypergraph
				
				% 4 edges
				\draw [fill=black] (2.5,0) circle (1pt);
				\draw [fill=black] (3.5,0) circle (1pt);
				\draw [fill=black] (2.5,1) circle (1pt);
				\draw [fill=black] (3.5,1) circle (1pt);
				
				\node (5) at (2.5,0) {};
				\node (6) at (3.5,0) {};
				\node (7) at (2.5,1) {};
				\node (8) at (3.5,1) {};
				
				\draw [line width=0.8pt] (2.5,0)-- (3.5,0);
				\draw [line width=0.8pt] (2.5,0)-- (2.5,1);
				\draw [line width=0.8pt] (2.5,1)-- (3.5,1);
				\draw [line width=0.8pt] (3.5,0)-- (3.5,1);
				
				% emptyhypergraph
				
				\draw [fill=black] (5,0) circle (1pt);
				\draw [fill=black] (6,0) circle (1pt);
				\draw [fill=black] (5,1) circle (1pt);
				\draw [fill=black] (6,1) circle (1pt);
				
				\node (9) at (5,0) {};
				\node (10) at (6,0) {};
				\node (11) at (5,1) {};
				\node (12) at (6,1) {};
				
				\draw [line width=0.8pt] (5,0)-- (6,1);
				\draw [line width=0.8pt] (5,1)-- (6,0);
				\draw [line width=0.8pt] (5,1)-- (6,1);
				\draw [line width=0.8pt] (6,1)-- (6,0);

				%5 edges
				\draw [fill=black] (7.5,0) circle (1pt);
				\draw [fill=black] (8.5,0) circle (1pt);
				\draw [fill=black] (7.5,1) circle (1pt);
				\draw [fill=black] (8.5,1) circle (1pt);
				
				\node (13) at (7.5,0) {};
				\node (14) at (8.5,0) {};
				\node (15) at (7.5,1) {};
				\node (16) at (8.5,1) {};
				
				\draw [line width=0.8pt] (7.5,0)-- (7.5,1);
				\draw [line width=0.8pt] (7.5,0)-- (8.5,1);
				\draw [line width=0.8pt] (8.5,0)-- (8.5,1);
				\draw [line width=0.8pt] (7.5,1)-- (8.5,0);
				\draw [line width=0.8pt] (7.5,0)-- (8.5,0);

				% 6 edges
				\draw [fill=black] (10,0) circle (1pt);
				\draw [fill=black] (11,0) circle (1pt);
				\draw [fill=black] (10,1) circle (1pt);
				\draw [fill=black] (11,1) circle (1pt);
				
				\node (17) at (10,0) {};
				\node (18) at (11,0) {};
				\node (19) at (10,1) {};
				\node (20) at (11,1) {};
				
				\draw [line width=0.8pt] (10,0)-- (11,0);
				\draw [line width=0.8pt] (10,0)-- (10,1);
				\draw [line width=0.8pt] (10,0)-- (11,1);
				\draw [line width=0.8pt] (10,1)-- (11,1);
				\draw [line width=0.8pt] (11,0)-- (11,1);
				\draw [line width=0.8pt] (10,1)-- (11,0);

				\begin{pgfonlayer}{bg}
					% 3 edges
					\draw[thick, fill = blue, fill opacity = .25] \hedgeiii{3}{2}{1}{3mm};
					\draw[thick, fill = red, fill opacity = .25] \hedgeiii{4}{2}{1}{3mm};
					\draw[thick, fill = green, fill opacity = .25] \hedgeiii{3}{4}{1}{3mm};
					\draw[thick, fill = yellow, fill opacity = .25] \hedgeiii{3}{4}{2}{3mm};
					
					%4 edges
					\draw[thick, fill = blue, fill opacity = .25] \hedgeiii{11}{10}{9}{3mm};
					\draw[thick, fill = red, fill opacity = .25] \hedgeiii{11}{12}{10}{3mm};
					
					\draw[thick, fill = blue, fill opacity = .25] \hedgeiii{15}{14}{13}{3mm};
					\draw[thick, fill = red, fill opacity = .25] \hedgeiii{16}{14}{13}{3mm};
					
					%6 edges
					\draw[thick, fill = blue, fill opacity = .25] \hedgeiii{19}{18}{17}{3mm};
					\draw[thick, fill = red, fill opacity = .25] \hedgeiii{20}{18}{17}{3mm};
					\draw[thick, fill = green, fill opacity = .25] \hedgeiii{19}{20}{17}{3mm};
					\draw[thick, fill = yellow, fill opacity = .25] \hedgeiii{19}{20}{18}{3mm};
				\end{pgfonlayer}
				
			\end{tikzpicture}
		
		\medskip
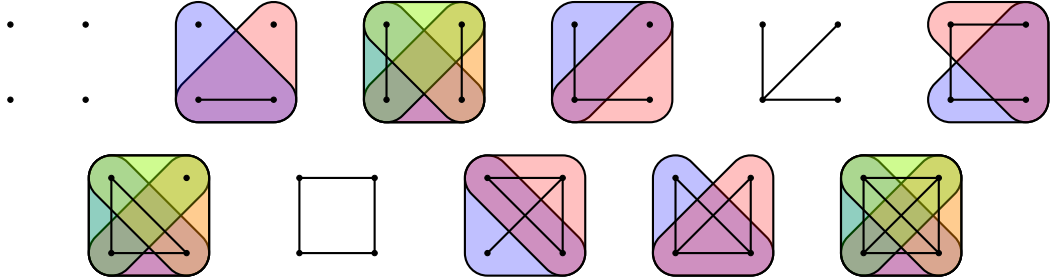
		Figure 3: All possible non-isomorphic link graphs and $3$-graphs induced by $G^H$ on $\{a,b,c,d\}$.
		\end{center}

	\end{proof}
	
	\begin{claim}
		\label{L34}
		If $L=\{4\}$ or $L=\{3,4\}$, then $f(n,L)=2^{\Theta(n^3)}$.
	\end{claim}
	\begin{proof}
	Denote $K_4^{3-}$ the tetrahedron minus one edge.
	For $L=\{3,4\}$, $f(n,L)$ counts the number of $K_4^{3-}$-free $n$-vertex 3-graphs and $f(n,\{4\})$ counts the number of $K_4^{3}$-free $n$-vertex 3-graphs. 
	For $K_4^{3-}$ it is known~\cite{K43-extremalfrankl,flagmatic} that 
		\begin{align*}
		    \frac{2}{7}\binom{n}{3} (1+o(1))\leq \textup{ex}(n,K_4^{3-})\leq 0.28689\binom{n}{3} (1+o(1)),
		\end{align*}
		where the lower bound comes from a construction by Frankl and F\"uredi~\cite{K43-extremalfrankl} and the upper bound from flag algebras~\cite{flagmatic}. For the tetrahedron it is known that 
		\begin{align*}
		    \frac{5}{9}\binom{n}{3} (1+o(1))\leq \textup{ex}(n,K_4^{3})\leq 0.5615\binom{n}{3} (1+o(1)),
		\end{align*}
		where the upper bound was proven by Baber~\cite{Baber} via the method of flag algebras. 
It follows that
		\begin{align*}
			2^{\frac{2}{7}\binom{n}{3} (1+o(1))}&\leq f(n,\{3,4\})=2^{\textup{ex}(n,K_4^{3-})(1+o(1))} \le 2^{0.28689\binom{n}{3}(1+o(1))} \\ \text{and} \quad \quad 2^{\frac{5}{9}\binom{n}{3} (1+o(1))}&\leq f(n,\{4\})=2^{\textup{ex}(n,K_4^3)(1+o(1))} \le 2^{0.5615\binom{n}{3}(1+o(1))},
		\end{align*}
where we used that the number of $n$-vertex $H$-free $3$-graphs is $2^{\textup{ex}(n,H)+o(n^3)}$, see~\cite{NSR06}. 
	\end{proof}

	\begin{claim}
		\label{L123}
		Let $n\geq 4$ and $L=\{1,2,3\}$. Then, $f(n,L)=2$.
	\end{claim}
	\begin{proof}
		Let $G$ be an $(L,4)$-free $3$-graph on $[n]$ with at least one edge. Let $C$ be a maximal clique in $G$ and suppose for contradiction that there is a vertex $v\notin C$. 
		Note that because $G$ contains an edge, $|C|\ge 3$.
		For any $3$ distinct vertices $i, j$ and $k$ in $C$, we have that $\{i, j, k, v\}$ induces a complete $3$-graph. Thus, $C\cup\{v\}$ is a clique as well, which contradicts the maximality of $C$. It follows that
		all the vertices in $[n]$ are in $C$, and hence $G$ is complete.
		We conclude that the only $(L,4)$-free $3$-graphs on $[n]$ are the complete $3$-graph and the $3$-graph with no edges.
	\end{proof}

	\begin{claim}
		\label{L023}
		Let $n\geq 5$ and $L=\{0,2,3\}$. Then, $f(n,L)=n+1$.
	\end{claim}
	\begin{proof}
		Let $G$ be an $(L,4)$-free $3$-graph on $[n]$. As $0 \in L$, the $3$-graph $G$ has at least one edge, hence we can fix a maximal clique $C$ in $G$.
		We claim that there is no edge containing exactly $2$ vertices of $C$.
		Suppose for contradiction that there is $v \notin C$ and $\{i,j\}\se C$
		such that $vij \in E(G)$. Then, for every $k \in C \setminus \{i,j\}$ the set $\{v,i,j,k\}$ spans at least $2$ edges.
		As $2, 3 \in L$, we have no other choice but to have a complete $3$-graph on $\{v,i,j,k\}$ for all $k \in C \setminus \{i,j\}$.
		As $vik$ is an edge for all $k \in C \setminus \{i,j\}$, we can 
		repeat the same argument and show that $\{v,i,k,\ell\}$ induces a complete graph for all $k,\ell \in C \setminus \{i\}$.
		We conclude that $vk\ell \in E(G)$ for all $k,\ell \in C$, and hence $C\cup \{v\}$ must be a clique. This contradicts the maximality of $C$.
		
		We now claim that $|V(G)\setminus C|\le 1$.
		Suppose for contradiction that there exists distinct vertices $i,j \notin C$.
		Let $c_1$ and $c_2$ be distinct vertices in $C$.
		As $c_1c_2i$ and $c_1c_2j$ are not edges of $G$ and $0,2 \in L$, we have that either $ijc_1$ or $ijc_2$ is an edge of $G$, otherwise we create a forbidden structure in $\{i,j,c_1,c_2\}$.
		Without loss of generality, suppose that $ijc_1 \in E(G)$.
		We can now take another vertex $c_3 \in C \setminus \{c_1,c_2\}$ and repeat the same argument.
		By analyzing the $3$-graph induced by $G$ on $\{i,j,c_2,c_3\}$, we conclude that either $ijc_2$ or $ijc_3$ is an edge of $G$.
		As we assumed that $ijc_2\notin E(G)$, we have $ijc_3 \in E(G)$ and hence the set $\{i,j,c_1,c_3\}$ induces exactly $2$ edges in $G$. Since $2 \in L$, this is a contradiction.

		If $|C|=n$, then $G$ is the complete $3$-graph. If $|C|=n-1$, then $G$ is the union of a clique of size $n-1$ and an isolated vertex.
		There are $n$ such $3$-graphs depending on which vertex the isolated vertex is. 
		We conclude that the number of $(L,4)$-free $3$-graphs on $[n]$ is $n+1$, for all $n \ge 5$.
	\end{proof}
	
	\section{Concluding Remarks}
	\label{sec:concluding remarkscount}
When analyzing the proof of Theorem~\ref{maincounting} in the case where $L=\{1,4\}$, $k=4$ and $r=3$, it can be observed that it actually gives
	\begin{align*}
		f(n,3,4,\{1,4\})\leq \prod_{m=1}^{n-1} m!=G(n+1)=2^{\frac{n^2}{2}\log n(1+o(1))},
	\end{align*}
	where $G$ is the Barnes $G$ function. On the other hand, 
	we have seen in the proof of Theorem~\ref{maincountingcorl} that
	\begin{align*}
	f(n,3,4,\{1,4\})\geq	|Q(n)| = 2^{\frac{n^2}{27}\log n (1+o(1))}.
	\end{align*}
	Towards solving Conjecture~\ref{conj:14-freecounting}, it would be interesting to first determine the constant in front of the main term of the exponent.
	\bibliographystyle{abbrv}
	\bibliography{bibilo}

%	\section{The lower bound - TO DO}
%	Take the Turan graph. Let us count all subhypergraphs of it without a forbidden structure. We cannot remove a edge with two vertices inside a part, otherwise we create a forbidden structure. The only edges we can remove are those which intersect all three parts. Moreover, we should remove a graph which is linear, otherwise we create a forbidden structure.
%	Let $L$ be the maximum linear hypergraph inside the Turan graph with only edges across. Note that $L$ is a subhypergraph of a steiner triple system $S(2,3,n)$.
%	We can just $3$-colour the vertices (in equitable partitions) of the steiner triple system and take the rainbow edges to form $L$.
%	How many edges will we have? For each edge, the probability of this edge be rainbow is $2/9$.
%	How many edges we have in $S(2,3,n)$? It's roughly $e(S)=\binom{n}{2}/3$.
%	It follows that $L$ has size
%	\begin{align*}
%		n^2/27 \le e(L) \le n^2/6.
%	\end{align*}
%	We can also choose edges greedily. At each step $i$ suppose we have chosen $i$ edges. Then we should avoid at most $3in$ edges so in step $i+1$ we have at least $n^3/27-3in$ choices left. It follows that we can proceed until $i = n^2/3^5$. In each step we have at least $n^3/3^4$ choices so it follows that the number of linear subhypergraphs is at least
%	\begin{align*}
%		\left(\dfrac{n^3}{3^4}\right)^{n^2/3^5} \Bigg / \left(\dfrac{n^2}{3^5}\right)! \sim n^{n^2/3^5}.
%	\end{align*}

\end{document}